\documentclass[11pt]{amsart}   
\usepackage{amsbsy,amsmath,amsthm,amssymb,color,verbatim}
\usepackage[backrefs]{amsrefs}
\usepackage{float}
\usepackage{fullpage,cancel}
\usepackage{hyperref}
\def\a{\alpha}
\def\l{\lambda}

\def\ZZ{{\mathbb Z}}

\def\A{\mathcal{A}}

\def\a{\alpha}

\def\fg{\mathfrak{g}}

\theoremstyle{definition}
\newtheorem{definition}{Definition}
\newtheorem{theorem}{Theorem}[section]
\newtheorem{proposition}{Proposition}[section] 
\newtheorem{lemma}{Lemma}[section] 
\newtheorem{corollary}{Corollary}[section]

\title{Kostant's Weight Multiplicity Formula and the\\ Fibonacci and Lucas Numbers}

\author{Kevin Chang}
\address{Department of Mathematics and Statistics, Williams College, United States}
\email{kc13@williams.edu}
\thanks{K. Chang was supported by a National Science Foundation grant (\#DMS1148695) through the Center for Undergraduate Research (CURM), Brigham Young University, and corporate sponsors.}

\author{Pamela E. Harris}
\address{Department of Mathematics and Statistics, Williams College, United States}
\email{peh2@williams.edu}
\thanks{P.\,E. Harris was supported by NSF award DMS-1620202.}

\author{Erik Insko}
\address{Department of Mathematics, Florida Gulf Coast University, United States}
\email{einsko@fgcu.edu}
\thanks{}

\keywords{}
\date{\today}
\begin{document}
\maketitle

    \begin{abstract}
Consider the weight $\l$ which is the sum of all simple roots of a simple Lie algebra. Using Kostant's weight multiplicity formula we describe and enumerate the contributing terms to the multiplicity of the zero weight in the representation with highest weight $\l$. 
We prove that in Lie algebras of type $A$ and $B$, the number of contributing terms to the multiplicity of the zero-weight space in the representation with highest weight $\l$ is given by a Fibonacci number, and that in Lie algebras of type $C$ and $D$, the analogous result is given by a multiple of a Lucas number.
    \end{abstract}


\section*{Introduction}

Let $G$ be a simple linear algebraic group over $\mathbb C$, $T$ a maximal algebraic torus in $G$ of dimension $r$, and $B$, $T\subseteq B \subseteq G$, a choice of Borel subgroup. Then let $\mathfrak g$, $\mathfrak h$, and $\mathfrak b$ denote the Lie algebras of $G$, $T$, and $B$ respectively. We let $\Phi$ denote the set of roots corresponding to $(\mathfrak {g,h})$, and $\Phi^+\subseteq\Phi$ is the set of positive roots with respect to $\mathfrak b$. Let $\Delta\subseteq\Phi^+$ be the set of simple roots. The denote the set of integral and dominant integral weights by $P(\mathfrak g)$ and $P_+(\mathfrak g)$ respectively. Let $W=Norm_G(T)/T$ denote the Weyl group corresponding to $G$ and $T$, and for any $w\in W$, we let $\ell(w)$ denote the length of $w$.

We recall that with a choice of a Cartan subalgebra it is well known that the finite-dimensional irreducible representations of a Lie algebra $\mathfrak{g}$ on the vector space $V$ can be studied by decomposing 
\begin{align}
V&=\oplus V_\a\label{direct}
\end{align} 
where the direct sum is indexed by a finite set of weights. Given a weight $\alpha$, the corresponding subspace $V_\a$ is called a weight space and the dimension of $V_\alpha$ is  called the multiplicity of $\alpha$. Thus to study representations of $\mathfrak{g}$ it suffices to determine the multiplicity of the weights appearing in~\eqref{direct}. For a more detailed account of this theory we refer the reader to \cite{GW}.

In this work we consider the weight $\l$ which is the sum of all simple roots of $\mathfrak{g}$.
We formally use Kostant's weight multiplicity formula to compute the multiplicity of the zero weight in the representation with highest weight $\lambda$, which we denote by $m(\lambda, 0)$.
This representation is the adjoint representation in the Lie algebra of type $A$ and the defining representation in type $B$;
 these cases were considered by Harris in \cite{PH} and \cite{Harris}, respectively. In the remaining Lie types it is a virtual representation: a representation arising from a non-dominant integral highest weight.
 
One way to compute the multiplicity of a weight $\mu$ is via Kostant's weight multiplicity formula~\cite{KMF}:
    \begin{align}
        m(\lambda,\mu)=\displaystyle\sum_{\sigma\in W}^{}(-1)^{\ell(\sigma)}\wp(\sigma(\lambda+\rho)-(\mu+\rho))\label{mult formula},
    \end{align}
where $W$ denotes the Weyl group of $\fg$, $\wp$ denotes Kostant's partition function, and $\rho=\frac{1}{2}\sum_{\alpha\in\Phi^+}\alpha$, with $\Phi^+$ denoting the set of positive roots of $\mathfrak{g}$. We recall that the  Weyl group is generated by reflections about hyperplanes lying perpendicular to the simple roots of the Lie algebra $\mathfrak{g}$, and for each $\sigma \in W$, the length $\ell(\sigma)$ represents the minimum number $k$ such that $\sigma$ is a product of $k$ reflections. Kostant's partition function  $\wp : \mathfrak h^* \rightarrow \ZZ$  is the nonnegative integer-valued function such that for each $ \xi \in \mathfrak h^*$, $\wp(\xi)$ counts the number of ways $\xi$ may be written as a nonnegative linear combination of positive roots.

A challenge in using Equation \eqref{mult formula} for weight multiplicity computations is the fact that the order of the Weyl group, indexing the sum, increases factorially as the rank of the Lie algebra considered increases. Additionally, many Weyl group elements contribute trivially to the alternating sum, thereby yielding another source of great inefficiency. In light of this, our work focuses on describing the elements of the Weyl group that contribute a nonzero term to the multiplicity formula, which leads to the following definition.
\begin{definition}
    For $\lambda,\mu$ dominant integral weights of $\mathfrak g$, we define the \emph{Weyl alternation set} by
    \begin{align}
        \mathcal A(\lambda,\mu)=\{\sigma\in W:\wp(\sigma(\lambda+\rho)-(\mu+\rho))>0\}.\label{KWMF}
    \end{align}
\end{definition}
The above definition implies that $\sigma\in W$ satisfies $\sigma\in\mathcal A(\lambda,\mu)$ if and only if $\sigma(\lambda+\rho)-(\mu+\rho)$ can be written as a nonnegative $\ZZ$-linear combination of positive roots. 

Harris, Insko, and Williams described and enumerated the Weyl alternation sets for the zero weight in the adjoint representation of the classical Lie algebras and showed that the cardinality of these sets is given by linear recurrences with constant coefficients \cites{H,HIW}. 
In addition, Harris, Lescinsky, and Mabie have provided visualizations for the Weyl alternation sets for different pairs of integral weights $\lambda$ and $\mu$ in the Lie algebra $\mathfrak{sl}_{3}(\mathbb C)$ \cites{H,HLM}.

Our research continues this work by describing and enumerating the elements of the Weyl alternation sets $\A(\l,0)$, where $\l$ is the sum of all the simple roots of a simple Lie algebra. We find that the cardinality of these Weyl alternation sets in the Lie algebras of type $A$ and $B$, are given by a Fibonacci number \cite{Harris} and in the Lie algebras of type $C$ and $D$, the analogous result is given by a multiple of a Lucas number. Our main results are summarized in Table \ref{tab:mainresults}, where $F_r$ and $L_r$ denote the $r^\text{th}$ Fibonacci and Lucas numbers, respectively.

\begin{table}[!htb]
    \caption{Summary of main results}\label{tab:mainresults}
    
    \begin{minipage}{.5\linewidth}
      \centering
\begin{tabular}{|c|c|c|}\hline
Classical Lie Algebras			&	$|\A(\l,0)|$\\\hline\hline
$A_r$ {\footnotesize $(r\geq 1)$}	&	$F_r$	\\\hline
$B_r$ {\footnotesize $(r\geq 2)$}	&	$F_{r+1}$\\\hline
$C_r$ {\footnotesize$(r\geq5)$}& $2L_{r-2}$\\\hline
$D_r$ {\footnotesize$(r\geq7)$}& 	$2L_{r-3}$\\\hline
\end{tabular}
    \end{minipage}%
    \begin{minipage}{.5\linewidth}
\centering
\begin{tabular}{|c|c|c|}\hline
Exceptional Lie Algebras			&	$|\A(\l,0)|$\\\hline\hline
$G_2$			& 	2\\\hline
$F_4$			& 	4\\\hline
$E_6$			& 	12\\\hline
$E_7$			& 	18\\\hline
$E_8$			& 	30\\\hline
\end{tabular}
    \end{minipage} 
\end{table}

These results give a glimpse into the complicated nature of weight multiplicity computations. Although the number of terms contributing nontrivially to $m(\l,0)$ is given by the Fibonacci and Lucas numbers, reducing the computation from a factorial number of terms, these numbers still grow exponentially, and one cannot reduce the computation any further.

This work is organized by considering a specific Lie algebra (in alphabetical order), providing needed background and proving the cardinality result for the Weyl alternation set involved. For more background in this area we point the interested reader to \cites{GW, Humphreys}. Lastly, we remark that the results of Table \ref{tab:mainresults} for the exceptional Lie algebras is a finite computation that was verified using the computer implementation presented in \cite{HIS}. 

\section{Lie algebra of type $A$}\label{sl}
In this section, we consider the Lie algebra $ \mathfrak{sl}_{r+1}(\mathbb C)$ for $r \geq 2$. In this case, the set of simple roots is given by $\Delta=\{\alpha_1,\alpha_{2}, \cdots, \alpha_{r}\}$,
and the set of positive roots is given by
$
\Phi^+ = \Delta \cup \{\alpha_i+\alpha_{i+1}+\cdots+\alpha_j: 1 \leq i < j \leq r\}.
$
The weight $\rho$ is defined as the half sum of the positive roots, $\rho=\frac{1}{2} \sum_{\alpha\in\Phi^+} \alpha$, which is equivalent to $\rho = \l+\varpi_2+\cdots+\varpi_r$, where $\l, \varpi_2, \ldots, \varpi_r$ are the fundamental weights of $\mathfrak{sl}_{r+1}(\mathbb{C})$. 
The Weyl group elements are generated by reflections about the hyperplanes that lie perpendicular to the simple roots $\alpha_i$. We denote these simple reflections by $s_i$, where $1 \leq i \leq r$, whose action on the simple roots is defined by
$
s_{i}(\alpha_{j})=\alpha_{j}$ if $|i-j|>1$, $s_{i}(\alpha_{j})=
    -\alpha_{j} $ if $i=j$, and $s_{i}(\alpha_{j})=
    \alpha_{i}+\alpha_{j}$ if $|i-j|=1$. 
The Weyl group elements act on the fundamental weights by
$s_{i}(\varpi_{j})=\varpi_j-\delta_{i,j}\alpha_i$, where $\delta_{i,j}=1$ when $i=j$ and 0 otherwise.
We now state the main result of this~section.

\begin{theorem}\label{setA}
    Let $\mathfrak g = \mathfrak{sl}_{r+1}(\mathbb{C})$ with $r \ge 2$. Then $\sigma \in \mathcal A(\l,0)$ if and only if $\sigma = 1$ or $\sigma = s_{i_{1}}s_{i_{2}}\cdots s_{i_{k}}$ for some collection of  nonconsecutive integers $2\leq i_{1},i_{2},\ldots,i_{k}\leq r-1$.
\end{theorem}
Theorem \ref{setA} first appeared in \cite[Proposition 2.1]{PH} and its proof used the fact that the Weyl group of $\mathfrak{sl}_{r+1}$ is isomorphic to the symmetric group $\mathfrak{S}_{r+1}$. Below we present a new proof using the fact that the Weyl group is generated by the root reflections $s_1,s_2,\ldots,s_r$. In particular, this proof technique illustrates the use of the root reflection action on $\lambda+\rho$,which provides us with a more direct style of proof..

\begin{proof}[Proof of Theorem \ref{setA}]
$(\Rightarrow)$ We prove this by establishing the contrapositive.
Suppose that $\sigma$ is neither the identity  nor $s_{i_{1}}s_{i_{2}}\cdots s_{i_{k}}$ for some nonconsecutive integers $2\leq i_{1},i_{2},\ldots,i_{k}\leq r-1$. Then $\sigma$ must contain $s_{1}$, or $s_{r}$, or $s_{i}s_{j}$ for consecutive integers $i$ and $j$. If $\sigma = s_{1}$, then we have that 
$
    s_{1}(\l+\rho) - \rho = s_{1}(\l)+s_{1}(\rho) - \rho = \l - \alpha_{1} + \rho - \alpha_{1} -\rho = \l - 2\alpha_{1},
$
which cannot be written as a sum of positive roots given the negative coefficient of $\alpha_{1}$. Hence, $s_{1} \notin \mathcal A(\l,0)$. Now \cite[Proposition 3.4]{HIW} shows that if  $\sigma\notin\A(\l,0)$, then neither is any $\sigma'$ containing $\sigma$ in 
its reduced word expression. Thus any $\sigma\in W$ containing $s_{1}$ in its reduced word expression cannot be in $\mathcal A(\l,0)$. Similarly, if $\sigma = s_{r}$, then we have that 
    $s_{r}(\l+\rho) - \rho = s_{r}(\l)+s_{r}(\rho) - \rho = \l - \alpha_{r} + \rho - \alpha_{r} -\rho = \l - 2\alpha_{r},
$
which cannot be written as a sum of positive roots because of the negative coefficient of $\alpha_{r}$. This implies that $s_{r} \notin \mathcal A(\l,0)$, and so any $\sigma$ containing $s_{r}$ in its reduced word expression is not in $\mathcal A(\l,0)$. 

Now suppose we have an arbitrary pair of consecutive integers $i$, $i+1$ such that $2 \leq i< r-1$. 
Using the property that the action of Weyl group elements on weights behave linearly, we have that $s_{i}s_{i+1}(\l+\rho)-\rho = s_{i}(\l + \rho - \alpha_{i+1})- \rho = (\l + \rho - \alpha_{i} - s_{i}(\alpha_{i+1})) - \rho = \l - 2\alpha_{i} - \alpha_{i+1},$
which cannot be written as a nonnegative $\ZZ$-linear combination of the positive roots and, thus, $s_is_{i+1}\notin\mathcal{A}(\l,0)$. Therefore any $\sigma$ containing $s_{i}s_{i+1}$ as a subword in its reduced word expression cannot be in $\mathcal A(\l,0)$. A similar argument shows that $s_{i+1}s_i\notin \mathcal A(\l,0)$. Thus,  if $\sigma \in \mathcal A(\l,0)$, then $\sigma = 1$ or $\sigma = s_{i_{1}}s_{i_{2}}\cdots s_{i_{k}}$ for some nonconsecutive integers $2\leq i_{1},i_{2},\ldots,i_{k}\leq r-1$, as claimed.

$(\Leftarrow)$ If $\sigma = 1$, then $1(\l+\rho) - \rho = \l$. Hence $1\in \mathcal A(\l,0)$.
If $\sigma = s_{i_{1}}s_{i_{2}}\cdots s_{i_{k}}$ for some nonconsecutive integers $2\leq i_{1},i_{2},\ldots,i_{k}\leq r-1$ we observe that
$
    s_{i_{1}}s_{i_{2}}\cdots s_{i_{k}}(\l+\rho)-\rho
    = \l - \sum_{j=1}^{k}\alpha_{i_j}
$
which can be written as a sum of positive roots. Thus $\sigma=s_{i_{1}}s_{i_{2}}\cdots s_{i_{k}}$ $\in\mathcal A(\l,0).$ \qedhere
\end{proof}

Before stating our next result, we recall that the Fibonacci numbers follow the recurrence $F_{r}=F_{r-1}+F_{r-2}$ with $F_{1}=F_{2}=1$.
\begin{corollary}\label{fibA}
    If $r \ge 2$ 
and $\l$ is the highest root of $\mathfrak {sl}_{r+1}$, then $|\mathcal A(\l,0)|=F_r,$ where $F_r$ denotes the $r^{th}$ Fibonacci number. 
\end{corollary}

The above result first appeared in \cite[Theorem 2.1]{PH}, but for sake of completeness we present a proof below, which uses the description of the elements of the Weyl group as products of root reflections. 

\begin{proof}[Proof of Corollary \ref{fibA}]
We proceed by induction. If $r = 2$, then by Theorem \ref{setA} we know $\mathcal A_{2}(\l,0) = \{1\}$, which shows that $|\mathcal A_{2}(\l,0)| = 1 = F_{2}$. If $r = 3$, then $\mathcal A_{3}(\l,0) = \{1, s_{2}\}$, which shows that $|\mathcal A_{3}(\l,0)| = 2 = F_{3}$.
Assume that for all $r$, with $3 \leq r \leq k$, $|\mathcal A_{r}(\l,0)| = F_{r}$. We consider the case when $r = k + 1$. Notice that all of the elements $\sigma \in W$ consisting of nonconsecutive products of the generators $s_{2}, s_{3}, \ldots ,s_{k}$ will either contain $s_{k}$ or not. If they do not contain $s_{k}$, then by our induction hypothesis, the number of Weyl group elements consisting of nonconsecutive products of the generators $s_{2}, s_{3}, \ldots  , s_{k-1}$ is given by $F_{k}$. If the Weyl group element contains $s_{k}$, then we must count the number of nonconsecutive products of the reflections $s_{2}, s_{3}, \ldots, s_{k-2}$, which by our induction hypothesis is given by $F_{k-1}$. Therefore $|\mathcal A_{k+1}(\l,0)| = F_{k-1} + F_{k} = F_{k+1}$. \qedhere
\end{proof}

\subsection{Nonzero weight spaces}
One could consider other nonzero weight spaces. In particular the case where $\mu$ is a positive root of the Lie algebra. For the Lioe algebra $\mathfrak{sl}_{r+1}(\mathbb{C})$, this was considered in the work of Harris \cite{PH}, where the following result was established
\begin{theorem}[Theorem 4.1 \cite{PH}]
If $\mu\neq 0$ is a dominant integral weight of $\mathfrak{al}_{r+1}(\mathbb{C})$ and $\l$ is the highest root, then $\mathcal{A} ( \l , \mu ) =\begin{cases}
\{1\}&\mbox{if $\mu=\l$}\\
\emptyset&\mbox{otherwise}.
    \end{cases}$
\end{theorem}

\section{Lie algebra of type $B$}\label{SO}
In this section, we consider the Lie algebra $\mathfrak{so}_{2r+1}(\mathbb C)$ for $r\geq 2$. For $1\leq i\leq r$, let $\varepsilon_i$ denote the $i^\text{th}$ standard basis vector in $\mathbb{R}^r$. If $\a_i=\varepsilon_i-\varepsilon_{i+1}$ for $1\leq i\leq r-1$ and $\a_r=\varepsilon_r$, then the set of simple roots of $\mathfrak{so}_{2r+1}(\mathbb{C})$ is given by $\Delta=\{\alpha_1,\ldots,\alpha_{r}\}$
and the set of positive roots is given by
$
\Phi^+=\{\varepsilon_i-\varepsilon_j,\varepsilon_i+\varepsilon_j:\;1\le i<j\le n\}\cup\{\varepsilon_i: 1\le i\le r\}.
$
The fundamental weights of $\mathfrak{so}_{2r+1}(\mathbb C)$ are defined by $\varpi_i=\varepsilon_1+\cdots+\varepsilon_i$ for $1\le i\le r-1$, $\varpi_r=\frac{1}{2}(\varepsilon_1+\varepsilon_2+\cdots+\varepsilon_r)$, and $\rho = \l+\cdots+\varpi_r$. Note that $   \l = \l=\alpha_{1} + \alpha_{2} + \cdots + \alpha_{r}  $.
 
The simple root reflections act on the simple roots and fundamental weights as follows. If $1\leq i\leq r-1$, then $s_i(\a_i)=-\a_i$, $s_{i}(\a_{i-1})=\a_{i-1}+\a_i$, $s_{i}(\a_{i+1})=\a_i+\a_{i+1}$, and $s_r(\a_r)=-\a_r$, $s_r(\a_{r-1})=\a_{r-1}+2\a_r$.
For any $1\leq i,j\leq r$, $s_i(\varpi_j)=\varpi_j-\delta_{i,j}\a_i$.

\begin{proposition}\label{prop:sionvarpi}
Let $\sigma=s_{i_1}s_{i_2}\cdots s_{i_k}$ where the indices of the simple reflections form a collection of nonconsecutive integers $2\leq i_1,\ldots,i_k\leq r$. Then $\sigma(\l+\rho)-\rho=\l-\sum_{j=1}^{k}\alpha_{i_j}$ is a nonnegative $\ZZ$-linear combination of positive roots.
\end{proposition}
\begin{proof}
Let $\sigma=s_{i_1}s_{i_2}\cdots s_{i_k}$ for some collection of  nonconsecutive integers $2\leq i_1,\ldots,i_k\leq r$. Note that $\sigma(\l)=\l$, and $\sigma(\rho)=\rho-\sum_{j=1}^{k}\alpha_{i_j}$. Thus
    $
        \sigma(\l+\rho)-\rho
        =\l-\sum_{j=1}^{k}\alpha_{i_j}
    $
which is a nonnegative $\ZZ$-linear combination of positive roots.
\end{proof}

\begin{theorem}\label{setB}
    Let $\mathfrak g = \mathfrak{so}_{2r+1}(\mathbb{C})$ with $r \ge 2$. Then $\sigma\in\mathcal A(\l,0)$ if and only if $\sigma=1$ or $\sigma=s_{i_1}s_{i_2}\cdots s_{i_k}$ for some nonconsecutive integers $2\leq i_1,\ldots,i_k\leq r$.
\end{theorem}
\begin{proof}

\noindent$(\Leftarrow)$ Let $\sigma = 1$, then $1(\l + \rho) - \rho = \l$ is a nonnegative $\ZZ$-linear combination of positive roots, thus $1 \in \mathcal A(\l,0)$, and Proposition \ref{prop:sionvarpi} implies that if $\sigma=s_{i_1}s_{i_2}\cdots s_{i_k}$ for some nonconsecutive integers $i_1,\ldots,i_k$ between and including $2$ and $r$, then $\sigma\in\mathcal{A}(\l,0)$.

$(\Rightarrow)$ Suppose $\sigma\in\mathcal A(\l,0)$. We proceed by induction on $\ell(\sigma)$. If $\ell(\sigma)=0$, then $\sigma=1$, which satisfies the needed condition. If $\ell(\sigma)=1$, then $\sigma=s_i$ for some $1 \le i \le r$. If $i=1$, then $s_{1}(\l + \rho) - \rho =  \l-2\alpha_{1}$, which implies $s_{1} \notin \mathcal A(\l,0)$, a contradiction. Thus, $\sigma\in\mathcal{A}(\l,0)$ cannot contain $s_{1}$ in its reduced word expression.
If $1<i\leq r$, then $s_{i}(\l + \rho) - \rho =\l-\alpha_{i}$, and   $s_i\in\mathcal{A}(\l,0)$ and $s_i$ is of the required form.

If $\ell(\sigma)=2$, then $\sigma=s_is_j$ for distinct integers $i,j$ satisfying $1 < i,j \leq r$. 
Without loss of generality, assume $i < j$. 
If $i,j$ are consecutive integers, then
$i = j-1$, with $1 < i,j < r$ or $i = r-1$ and $j = r$. In either case we note
    $    s_{j-1}s_{j}(\l + \rho) - \rho 
        = \l-\alpha_{i}-2\alpha_{j}$ and 
$    s_{r-1}s_{r}(\l + \rho) - \rho = \l-\alpha_{r-1}-3\alpha_{r}$
none of which can be written as a nonnegative $\ZZ$-linear combination of positive root. Thus, $s_{r-1}s_r$, $s_{r}s_{r-1}$, $s_{j-1}s_j$, $s_{j}s_{j-1} \notin \mathcal A(\l,0)$, a contradiction.
Moreover, any $\sigma\in W$ containing $s_{j}s_{j-1}$ or $s_{j-1}s_{j}$ in its reduced word expression cannot be in $\mathcal A(\l,0)$ for all $2<j\leq r$. The  case were $i,j$ are consecutive was already considered in Proposition \ref{prop:D1}.


Suppose that for all $\sigma \in \mathcal A(\l,0)$ with $1 < \ell(\sigma) \leq k$, there exists some nonconsecutive integers $2\leq i_1,\ldots,i_{\ell(\sigma)}\leq r$ such that $\sigma=s_{i_1}s_{i_{2}}\cdots s_{i_\ell(\sigma)}$.
Now consider $\tau\in\mathcal{A}(\l,0)$ with $\ell(\tau)= k+1$. Then $\tau=s_{l}\sigma$ for some $2\leq l\leq r$ and for some $\sigma\in W$ with $\ell(\sigma)=k$. Note that in fact $\sigma\in \mathcal{A}(\l,0)$, as otherwise $\tau$ would not be in $\mathcal{A}(\l,0)$, giving a contradiction. Hence, by our induction hypothesis there exist nonconsecutive integers $2\leq i_{1},i_{2},\cdots,i_{k}\leq r$ such that $\sigma = s_{i_1}\cdots s_{i_{k}}$. By Proposition~\ref{prop:sionvarpi},  $\sigma(\l+\rho) = \l+\rho-\sum_{j=1}^{k}\alpha_{i_j}$. Hence
    $
        \tau(\l + \rho) - \rho = s_{l}\sigma(\l + \rho) - \rho
        = \l - \alpha_{l} - \sum_{j=1}^{k}s_{l}(\alpha_{i_j})=\l-\alpha_{l}-\sum_{j=1}^{k}(\alpha_{i_j}+c_{l,i_j}\a_{l})$
where $c_{l,i_j}=2$ if $i_l=r$ and $i_j=r-1$, 
            $c_{l,i_j}=0 $ if $|l-i_j|>1$ and $c_{l,i_j}=1$ otherwise.
Observe that whenever $c_{l,j_{1}}=1$ or $2$, the expression $\tau(\l+\rho)-\rho$ contains a negative coefficient on a simple root, and thus $\tau \notin \mathcal A(\l,0)$, a contradiction. Therefore, $l,i_1,\cdots,i_{k}$ must be nonconsecutive integers between and including $2$ and~$r$.
\end{proof}

\begin{corollary}\label{fibB}
    If $r \ge 2$ and $\l = \alpha_{1} + \alpha_{2} + \cdots + \alpha_{r}$ is a fundamental weight of $\mathfrak {so}_{2r+1}(\mathbb{C})$, then $|\mathcal A(\l,0)|=F_{r+1}$, where $F_{r+1}$ denotes the $(r+1)^{th}$ Fibonacci number.
\end{corollary}

The proof of Corollary \ref{fibB} is analogous to that of Corollary \ref{fibA}, hence we omit it.

We remark that the results in this section first appeared in an unpublished preprint of the second author as \cite[Proposition 2.1, Theorem 2.1, and Theorem 1.1]{Harris}, respectively. However, the proofs presented in this current manuscript are new and, as in the previous section, they use the action of root reflections on $\l+\rho$ without using the definition of the root reflections involving the symmetric bilinear form 
on $\mathfrak{h}^*$ corresponding to the trace form as in \cite{GW}.
\subsection{Nonzero weight spaces}
We now consider the case when $\mu$ is a nonzero dominant weight of $\mathfrak{so}_{2r+1}$ and compute the Weyl alternation sets $\mathcal A(\l,\mu)$. Throughout this section $r\ge 2$ and as before $\l=\a_1+\a_2+\cdots+\a_r$.

\begin{theorem}\label{nonzeroweights}If $\mu\in P_+(\mathfrak{so}_{2r+1})$ and $\mu\neq 0$, then $\mathcal A(\l,\mu)=\begin{cases}\{1\}&\text{if $\mu=\l$}\\\emptyset&\text{otherwise.}\end{cases}$
\end{theorem}
We begin by proving the following technical results from which Theorem \ref{nonzeroweights} follows.
\begin{proposition}\label{altsetidentity}If $\l=\sum_{\alpha\in\Delta}\alpha$ is a fundamental weight of $\mathfrak {so}_{2r+1}$, then $\mathcal A(\l,\l)=\{1\}$.
\end{proposition}

\begin{proof}Since $\l=\alpha_1+\cdots+\alpha_r$, notice $\sigma(\l+\rho)-\rho-\l$ is a nonnegative integral sum of positive roots only if $\sigma(\l+\rho)-\rho$ is. By Theorem \ref{setB} we know $\sigma(\l+\rho)-\rho$ is a nonnegative $\mathbb{Z}$-linear combination of positive roots if and only if $\sigma=s_{i_1}s_{i_2}\cdots s_{i_k}$, for some nonconsecutive integers $i_1,\ldots,i_k$ between $2$ and $r$. Hence $\mathcal A(\l,\l)\subset\mathcal A(\l,0)$. Suppose that $\sigma\in\mathcal A(\l,\l)$ with $\ell(\sigma)=k\ge 1$, then there exist nonconsecutive integers $i_1,\ldots,i_k$ between $2$ and $r$ such that $\sigma=s_{i_1}s_{i_2}\cdots s_{i_k}$. By Proposition \ref{prop:sionvarpi} we have that $\sigma(\l+\rho)-\rho=\l-\sum_{j=1}^{k}\alpha_{i_j}$. Then notice $\sigma(\l+\rho)-\rho-\l$ will not be a nonnegative integral sum of positive roots, reaching a contradiction. Thus $\ell(\sigma)=0$ and $\sigma=1$.
\end{proof}

\begin{proposition}Let $\mu\in P_+(\mathfrak{so}_{2r+1})$, and $\mu\neq 0$. Then there exists $\sigma\in W$ such that
$\wp(\sigma(\l+\rho)-\rho-\mu)>0$ if and only if $\mu=\l$.
\end{proposition}
\begin{proof}
$(\Rightarrow)$ Let $\mu\in P_+(\mathfrak{so}_{2r+1})$ with $\mu\neq 0$, and assume $\sigma\in W$ such that
$\wp(\sigma(\l+\rho)-\rho-\mu)>0$. By \cite[Proposition 3.1.20]{GW}, we know that $P_+(\mathfrak{so}_{2r+1})$ consists of all weights $\mu=k_1\varepsilon_1+k_2\varepsilon_2+\cdots+k_r\varepsilon_r$, with $k_1\ge k_2\ge \cdots\ge k_r\ge 0$. Here $2k_i$, and $k_i-k_j$ are integers for all $i,j$.

Now observe that $\sigma(\l+\rho)-\rho-\mu=\sigma((r+\frac{1}{2})\varepsilon_1+(r-\frac{3}{2})\varepsilon_2+(r-\frac{5}{2})\varepsilon_3+\cdots+\frac{1}{2}\varepsilon_{r-1}+\varepsilon_r)-((r-\frac{1}{2})\varepsilon_1+(r-\frac{3}{2})\varepsilon_2+\cdots+\frac{1}{2}\varepsilon_r)-(k_1\varepsilon_1+\cdots+k_r\varepsilon_r)$.
Let $a_i$ denote the coefficient of $\alpha_i$ in $\sigma(\l+\rho)-\rho-\mu$. Then $a_1=\begin{cases}-i+1-k_1&\text{if $\sigma(\varepsilon_1)=\varepsilon_i$ for $2\le i\le r$}\\-2r+i-k_1&\text{if $\sigma(\varepsilon_1)=-\varepsilon_i$ for $2\le i\le r$}\\1-k_1&\text{if $\sigma(\varepsilon_1)=\varepsilon_1$}\\-2r-k_1&\text{if $\sigma(\varepsilon_1)=-\varepsilon_1$.}\end{cases}$

Since $r\ge 2$ and $a_1\in\mathbb N$, we have that $\sigma(\varepsilon_1)=\varepsilon_1$ and $a_1=1-k_1$. If $k_1=0,$ then $k_i=0$ for all $1\le i\le r$, and so $\mu=0$, a contradiction. Hence $k_1=1$. Since $k_i-k_j\in\mathbb Z$ for all $i$ and $j$, and since $1=k_1\ge k_2\ge k_3\ge \cdots\ge k_r\ge 0$, we have that $k_i=0$ or $1$, for all $2\le i\le r$. We want to show that $k_i=0$ for all $2\le i\le r$. It suffices to show $k_2=0$. A simple computation shows that $a_2=\begin{cases}-i+2-k_2&\text{if $\sigma(\varepsilon_2)=\varepsilon_i$ for $3\le i\le r$}\\-2r+i+1-k_2&\text{if $\sigma(\varepsilon_2)=-\varepsilon_i$ for $3\le i\le r$}\\-k_2&\text{if $\sigma(\varepsilon_2)=\varepsilon_2$}\\-2r+3-k_2&\text{if $\sigma(\varepsilon_2)=-\varepsilon_2$}.\end{cases}$

Since $r\ge 2$ and $a_2\in\mathbb N$, we have that $\sigma(\varepsilon_2)=\varepsilon_2$, and hence $k_2=0$. Thus $\mu=\varepsilon_1=\l$.

$(\Leftarrow)$ By Proposition \ref{altsetidentity}, we know if $\mu=\l$, then $\wp(\sigma(\l+\rho)-\rho-\l)>0$ when $\sigma=1$.
\end{proof}

\begin{theorem}If $\mu\in P(\mathfrak{so}_{2r+1})$, then $m(\l,\mu)=\begin{cases}1&\text{if $\mu=0$ or $\mu\in W\cdot \l$}\\\emptyset&\text{otherwise.}\end{cases}$
\end{theorem}
\begin{proof}
Recall that given $\mu\in P(\mathfrak {so}_{2r+1})$, there exists $w\in W$ and $\xi\in P_+(\mathfrak {so}_{2r+1})$ such that $w(\xi)=\mu$ and also recall that weight multiplicities are invariant under $W$ \cite[Propositions 3.1.20, 3.2.27]{GW}. Thus it suffices to consider $\mu\in P_{+}(\mathfrak{so}_{2r+1})$. Corollary \ref{multofzeroweight} gives $m(\tilde\alpha,0)=1$, while Theorem \ref{nonzeroweights} implies $m(\l,\l)=1$ and $m(\l,\mu)=0$, whenever $\mu\in P_+(\mathfrak {so}_{2r+1})\setminus\{0,\l\}$.
\end{proof}

\section{Lie algebra of type $C$}\label{SP}
In this section, we consider the Lie algebra $\mathfrak{sp}_{2r}(\mathbb C)$ for $r\geq 3$. For $1\leq i\leq r$ let $\varepsilon_i$ denote the $i^\text{th}$ standard basis vector in $\mathbb{R}^r$. If $\a_i=\varepsilon_i-\varepsilon_{i+1}$ for $1\leq i\leq r-1$ and $\a_r=2\varepsilon_r$, then the set of simple roots of $\mathfrak{sp}_{2r}(\mathbb{C})$ is given by $\Delta=\{\alpha_1,\ldots,\alpha_{r}\}$
and the set of positive roots is given by $\Phi^+=\{\varepsilon_i-\varepsilon_j,\varepsilon_i+\varepsilon_j:\;1\le i<j\le r\}\cup\{2\varepsilon_i: 1\le i\le r\}$.
The fundamental weights of $\mathfrak{sp}_{2r}(\mathbb C)$ are  $\varpi_i=\varepsilon_1+\cdots+\varepsilon_i$ for $1\le i\le r$, and $\rho= \l+\cdots+\varpi_r$.
The simple root reflections act on the simple roots and fundamental weights as follows. If $1\leq i\leq r$, then
$
s_{i}(\alpha_{j})=    \alpha_{j}$ if $|i-j|>1$, $s_{i}(\alpha_{j})=-\alpha_{j}$ if $i=j$, $s_{i}(\alpha_{j})=
    \alpha_{i}+\alpha_{j}$ if $|i-j|=1$ and $i\neq r-1$, $j\neq r$, and $s_{r-1}(\alpha_{r})=
    2\alpha_{r-1}+\alpha_{r}$.
As before $s_i(\varpi_j)=\varpi_j-\delta_{i,j}\alpha_i$ for all $1\leq i,j\leq r$.
Throughout this section, we let  $\lambda = \alpha_{1} + \alpha_{2} + \cdots +\alpha_{r}$.

\begin{proposition}\label{prop:sionlambda}
Let $\sigma=s_{i_1}s_{i_2}\cdots s_{i_k}$ for some nonconsecutive integers $2\leq i_1,\ldots,i_k\leq r-1$. If $\sigma$ contains $s_{r-1}$ in its reduced word expression, then $\sigma(\lambda+\rho)-\rho=\lambda+ 2\alpha_{r-1}-\sum_{j=1}^{k}\alpha_{i_j} $, otherwise $\sigma(\lambda+\rho)-\rho=\lambda-\sum_{j=1}^{k}\alpha_{i_j}$, both of which are nonnegative $\ZZ$-linear combinations of positive roots.
\end{proposition}
\begin{proof}
Let $\sigma=s_{i_1}s_{i_2}\cdots s_{i_k}$ for some nonconsecutive integers $2\leq i_1,\ldots,i_k\leq r-1$. If $\sigma$ contains $s_{r-1}$, without loss of generality, let $i_{k} = r-1$, and observe that
$
    \sigma(\lambda+\rho)-\rho=s_{i_1}s_{i_2}\cdots s_{i_k-1}s_{r-1}(\lambda + \rho) - \rho
    =s_{i_1}s_{i_2}\cdots s_{i_k-1}(\lambda + \rho) - \rho
    =\lambda+\alpha_{r-1}-\sum_{j=1}^{k-1}\alpha_{i_j}=\lambda+2\alpha_{r-1}-\sum_{j=1}^{k}\alpha_{i_j}=.
$
If $\sigma$ does not contain $s_{r-1}$,
then
$
    \sigma(\lambda+\rho)-\rho=s_{i_1}s_{i_2}\cdots s_{i_k}(\lambda + \rho) - \rho =\lambda+\rho-\rho-\sum_{j=1}^{k}\alpha_{i_j} =\lambda-\sum_{j=1}^{k}\alpha_{i_j}.
$
Lastly, note that both expressions can be written as nonnegative integrals sum of positive~roots.
\end{proof}

\begin{proposition}\label{prop:4.2}
If $\sigma=s_{i_1}s_{i_2}\cdots s_{i_k}$ for some nonconsecutive integers $2\leq i_1,\ldots,i_k\leq r-4$, then
\begin{itemize}
    \item  $\sigma s_{r-2}s_{r-1}(\lambda+\rho)-\rho=\lambda-\left(\sum_{j=1}^{k}\alpha_{i_j}\right)-\alpha_{r-2}$
    \item  $\sigma s_{r-1}s_{r-2}(\lambda+\rho)-\rho=\sigma s_{r-2}s_{r-1}s_{r-2}(\lambda+\rho)-\rho=\lambda-\left(\sum_{j=1}^{k}\alpha_{i_j}\right)-\alpha_{r-2}-\alpha_{r-1}$
\end{itemize}
all of which can be represented as nonnegative $\ZZ$-linear combinations of positive roots.
\end{proposition}
\begin{proof} The result follows from Proposition \ref{prop:sionlambda} and by computing the action of the simple roots $s_{r-2}$ and $s_{r-1}$ on $\lambda+\rho$.
\end{proof}
The following result describes all of the elements of $\mathcal A(\lambda,0)$ for the Lie algebra of type~$C$.
\begin{theorem}\label{setC}
    Let $\mathfrak g = \mathfrak{sp}_{2r}(\mathbb{C})$ with $r \ge 3$. Then $\sigma \in \mathcal A(\lambda,0)$ if and only if 
    \begin{enumerate}
        \item $\sigma = 1$ or
        \item $\sigma = s_{i_{1}}s_{i_{2}}\cdots s_{i_{k}}$ for some nonconsecutive integers $2\leq i_{1},i_{2},\ldots,i_{k}\leq r-1$ or
        \item $\sigma = s_{i_{1}}s_{i_{2}}\cdots s_{i_{k}}\pi$ for some nonconsecutive integers $2\leq i_{1},i_{2},\ldots,i_{k}\leq r-4$ and
        $\pi \in \{s_{r-2}s_{r-1}, s_{r-1}s_{r-2}, s_{r-2}s_{r-1}s_{r-2}\}$.
    \end{enumerate}
\end{theorem}

\begin{proof}
\noindent$(\Leftarrow)$  Let $\sigma = 1$, then $1(\lambda + \rho) - \rho = \lambda$ is a nonnegative $\ZZ$-linear combination of positive roots, thus $1 \in \mathcal A(\lambda,0)$.  Propositions \ref{prop:sionlambda} and \ref{prop:4.2} show that if $\sigma$ is of the form listed in (2) or (3) above, then
then $\sigma \in  \mathcal A(\lambda,0)$.

$(\Rightarrow)$ Suppose $\sigma\in W$ is not of the three forms listed above. Then $\sigma$ contains $s_{1}$ or $s_{r}$, or
    $s_{i}s_{j}$ where $i, j$ are consecutive integers, but not of the forms $s_{r-2}s_{r-1}$ or $s_{r-1}s_{r-2}$.
We observe that $s_{1}(\lambda + \rho) - \rho = ( \lambda - \alpha_{1} + \rho - \alpha_{1} ) - \rho = \lambda-2\alpha_{1}$ and $s_{r}(\lambda + \rho) - \rho = ( \lambda - \alpha_{r} + \rho - \alpha_{r} ) - \rho = \lambda-2\alpha_{r}$, which cannot be written as a sum of positive roots because of the negative coefficient of $\alpha_{1}$ and of $\alpha_{r}$, respectively. This implies that $s_{1}, s_{r} \notin \mathcal A(\lambda,0)$, and hence if $\sigma$ contains $s_{1}$ or $s_{r}$ in its reduced word expression, then $\sigma \notin \mathcal A(\lambda,0)$.

For consecutive integers $1 < j-1,j < r-1$ we have 
$
    s_{j-1}s_{j}(\lambda + \rho) - \rho 
    = \lambda - 2\alpha_{j-1} - \alpha_{j}$ and
$    s_{j}s_{j-1}(\lambda + \rho) - \rho  
    = \lambda -\alpha_{j-1}-2\alpha_{j}$,
which implies that $s_{j-1}s_j,s_{j}s_{j-1} \notin \mathcal A(\lambda,0)$. Hence if $\sigma$ contains $s_i,s_j$ for some consecutive integers $2\leq i,j\leq r-2$ then $\sigma\notin\A(\lambda,0)$. Thus $\sigma$ must be of one of the three forms listed in the theorem in order for $\sigma \in \mathcal A(\lambda,0)$.
\end{proof}

Recall that the Lucas numbers follow the recurrence $L_{r} = L_{r-1} + L_{r-2}$, with $L_{1} = 1$ and $L_{2} = 3$. We can now connect our work with this famous sequence of integers.

\begin{corollary}\label{LucasC}
    If $r \ge 3$ and $\lambda=\a_1+\a_2+\cdots+\a_r$ is a weight of $\mathfrak {sp}_{2r}(\mathbb{C})$, then $|\mathcal A(\lambda,0)|=2L_{r-2}$, where $L_{k}$ denotes the $k^{th}$ Lucas number. 
\end{corollary}

\begin{proof}
As in Corollary \ref{fibA}, we know that there are $F_{r}$ Weyl group elements in $\A(\lambda,0)$ arising from parts 1 and 2 of Theorem \ref{setC}. By the same reasoning, there are $F_{r-3}$ elements $\sigma = s_{i_{1}}s_{i_{2}}\cdots s_{i_{k}}\pi$ for some nonconsecutive integers $2\leq i_{1},i_{2},\ldots,i_{k}\leq r-4$, for each $\pi$ as specified in part 3 of Theorem \ref{setC}. This yields an additional $3F_{r-3}$ elements in $\A(\lambda,0)$. Thus $|\mathcal A(\lambda,0)|=F_{r}+3F_{r-3}$, where $F_{k}$ denotes the $k^{th}$ Fibonacci number. The result follows from the fact that $F_{r}+3F_{r-3} = 2L_{r-2}$.
\end{proof}

\subsection{Nonzero weight spaces}
Throughout this section $r\ge 2$ and as before $\l=\a_1+\a_2+\cdots+\a_r$.We now consider the case when $\mu$ is a nonzero dominant weight of $\mathfrak{sp}_{2r}$ and compute the Weyl alternation sets $\mathcal A(\l,\mu)$. 

\begin{theorem}\label{nonzeroweightsC}If $\mu\in P_+(\mathfrak{sp}_{2r})$ and $\mu\neq 0$, then $\mathcal A(\l,\mu)=\begin{cases}\{1\}&\text{if $\mu=\l$}\\\emptyset&\text{otherwise.}\end{cases}$
\end{theorem}
We begin by proving the following technical results from which Theorem \ref{nonzeroweightsC} follows.
\begin{proposition}\label{altsetidentityC}If $\l=\sum_{\alpha\in\Delta}\alpha$ is weight of $\mathfrak {sp}_{2r}$, then $\mathcal A(\l,\l)=\{1\}$.
\end{proposition}

\section{Lie algebra of type $D$}\label{SO2}
In this section, we consider the Lie algebra $\mathfrak g=\mathfrak{so}_{2r}(\mathbb C)$  for $r\geq 4$. For $1\leq i\leq r$ let $\varepsilon_i$ denote the $i^\text{th}$ standard basis vector in $\mathbb{R}^r$. If $\a_i=\varepsilon_i-\varepsilon_{i+1}$ for $1\leq i\leq r-1$ and $\a_r=\varepsilon_{r-1}+\varepsilon_r$, then the set of simple roots is given by $\Delta=\{\alpha_1,\ldots,\alpha_{r}\}$
and the set of positive roots is given by $\Phi^+=\{\varepsilon_i-\varepsilon_j,\varepsilon_i+\varepsilon_j:\;1\le i<j\le r\}.
$
The fundamental weights of $\mathfrak{sp}_{2r}(\mathbb C)$ are  $\varpi_i=\varepsilon_1+\cdots+\varepsilon_i$ for $1\le i\le r-2$, $\varpi_{r-1}=\frac{1}{2}(\varepsilon_1+\cdots+\varepsilon_{r-1}-\varepsilon_r)$, $\varpi_r=\frac{1}{2}(\varepsilon_1+\cdots+\varepsilon_{r-1}+\varepsilon_r)$, and $\rho= \l+\cdots+\varpi_r$.
The simple root reflections act on the simple roots and fundamental weights as follows. If $1\leq i\leq r$, then $s_i(\a_i)=-\a_i$. 
If $1 \leq i < j \leq r - 1$ with $|i - j| = 1$ or if $i = r - 2$ and $j = r$, then
$s_i(\a_j) = s_j(\a_i) = \a_i + \a_j$. Lastly, $s_{r-1}(\a_r) = \a_r$, $s_r(\a_{r-1}) = \a_{r-1}$, and in all other cases $s_i(\a_j)=\a_j$.
As before $s_i(\varpi_j)=\varpi_j-\delta_{i,j}\alpha_i$ for all $1\leq i,j\leq r$.
Throughout this section, we let  $\lambda = \alpha_{1} + \alpha_{2} + \cdots +\alpha_{r}$.

\begin{proposition}\label{prop:D1}
Let $\sigma=s_{i_1}s_{i_2}\cdots s_{i_k}$ for some nonconsecutive integers $2\leq i_1,\ldots,i_k\leq r-2$. If $\sigma$ contains $s_{r-2}$, then $\sigma(\lambda+\rho)-\rho=\lambda+\a_{r-2}-\sum_{j=1}^{k}\alpha_{i_j}$, otherwise $\sigma(\lambda+\rho)-\rho=\lambda-\sum_{j=1}^{k}\alpha_{i_j}$, both of which are nonnegative $\ZZ$-linear combinations of positive roots.
\end{proposition}
\begin{proof}
Let $\sigma=s_{i_1}s_{i_2}\cdots s_{i_k}$ for some nonconsecutive integers $2\leq i_1,\ldots,i_k\leq r-2$. If $\sigma$ contains $s_{r-2}$, then without loss of generality assume $i_k=r-2$ and note
$
        \sigma(\lambda+\rho)-\rho=s_{i_1}s_{i_2}\cdots s_{i_{k-1}}s_{r-2}(\lambda + \rho) - \rho =s_{i_1}s_{i_2}\cdots s_{i_{k-1}}(\lambda + \rho) - \rho
        =\lambda-\sum_{j=1}^{k-1}\alpha_{i_j}=\lambda+\a_{r-2}-\sum_{j=1}^{k}\alpha_{i_j}.
$
However, if $\sigma$ does not contain $s_{r-2}$, then
$
        \sigma(\lambda+\rho)-\rho = s_{i_1}s_{i_2}\cdots s_{i_k}(\lambda + \rho) - \rho
        =\lambda-\sum_{j=1}^{k}\alpha_{i_j}.
    $
Lastly, note that both expressions can be written as a nonnegative $\ZZ$-linear combination of positive roots.
\end{proof}

\begin{proposition}\label{prop:D2}
If $\sigma=s_{i_1}s_{i_2}\cdots s_{i_k}$ for some nonconsecutive integers  $2\leq i_1,\ldots,i_k\leq r-5$, then 
\begin{itemize}
    \item  $\sigma s_{r-3}s_{r-2}(\lambda+\rho)-\rho=\lambda-\left(\sum_{j=1}^{k}\alpha_{i_j}\right)-\alpha_{r-3}$,
    \item  $\sigma s_{r-2}s_{r-3}(\lambda+\rho)-\rho=\sigma s_{r-3}s_{r-2}s_{r-3}(\lambda+\rho)-\rho=\lambda-\left(\sum_{j=1}^{k}\alpha_{i_j}\right)-\alpha_{r-3}-\alpha_{r-2}$,
\end{itemize}
both of which can be written as nonnegative $\ZZ$-linear combinations of positive roots.
\end{proposition}
\begin{proof} The result follows from Proposition \ref{prop:D1} and by computing the action of the simple roots $s_{r-3}$ and $s_{r-2}$ on $\lambda+\rho$.
\end{proof}

\begin{theorem}\label{setD}
    Let $\mathfrak g = \mathfrak{so}_{2r}(\mathbb{C})$ with $r \ge 4$. Then $\sigma \in \mathcal A(\lambda,0)$ if and only if 
    \begin{enumerate}
        \item $\sigma = 1$ or
        \item $\sigma = s_{i_{1}}s_{i_{2}}\cdots s_{i_{k}}$ for some nonconsecutive integers $2\leq i_{1},i_{2},\ldots,i_{k}\leq r-2$ or
        \item $\sigma = s_{i_{1}}s_{i_{2}}\cdots s_{i_{k}}\pi$ for some nonconsecutive integers $2\leq i_{1},i_{2},\ldots,i_{k}\leq r-5$ and
        $
        \pi \in\{s_{r-3}s_{r-2},\;
            s_{r-2}s_{r-3},\;
            s_{r-3}s_{r-2}s_{r-3}   
        \}
        .$
    \end{enumerate}
\end{theorem}

\begin{proof}
\noindent$(\Leftarrow)$ Let $\sigma = 1$, then $1(\lambda + \rho) - \rho = \lambda$ is a nonnegative $\ZZ$-linear combination of positive roots. Hence $1 \in \mathcal A(\lambda,0)$. If $\sigma \in W$ has one of the forms listed in (2) or (3), then Propositions \ref{prop:D1} and \ref{prop:D2} show that $\sigma \in \mathcal{A}(\lambda,0)$. 

$(\Rightarrow)$ Suppose $\sigma\in W$ is not of the three forms listed above. Then $\sigma$ contains $s_{1}$, $s_{r-1}$, $s_{r}$, or consecutive reflections  $s_{i}$ and $s_{j}$, where $\{i,j\} \neq \{r-3,r-2\}$. Note that
$s_{1}(\lambda + \rho) - \rho = \lambda-2\alpha_{1}$,
        $s_{r-1}(\lambda + \rho) - \rho 
        = \lambda-2\alpha_{r-1}$, and $
        s_{r}(\lambda + \rho) - \rho 
        = \lambda-2\alpha_{r}$,
none of which can be written as sums of positive roots because of the negative coefficients of $\alpha_{1}$, $\alpha_{r-1}$ and $\alpha_{r}$, respectively. This implies that $s_{1},s_{r-1},s_{r} \notin \mathcal A(\lambda,0)$, and hence if $\sigma$ contains $s_{1}$, $s_{r-1}$, or $s_{r}$ in its reduced word expression, then $\sigma \notin \mathcal A(\lambda,0)$. 

For consecutive integers $2 \leq j-1,j \leq r-3$ we have 
$s_{j-1}s_{j}(\lambda + \rho) - \rho= \lambda - 2\alpha_{j-1} - \alpha_{j}$ and
$    s_{j}s_{j-1}(\lambda + \rho) - \rho = \lambda - \alpha_{j-1} - 2\alpha_{j}$,
which implies that $s_{j-1}s_j,\;s_{j}s_{j-1} \notin \mathcal A(\lambda,0)$.
Hence if $\sigma$ contains $s_i,s_j$ for some consecutive integers $2\leq i,j\leq r-3$ then $\sigma\notin\A(\lambda,0)$. Thus $\sigma$ must be of one of the three forms listed in the theorem in order for $\sigma\in\A(\lambda,0)$.
\end{proof}

\begin{corollary}\label{LucasD}
     If $r \ge 4$ and $\lambda=\a_1+\a_2+\cdots+\a_r$ is a weight of $\mathfrak {so}_{2r}(\mathbb{C})$, then $|\mathcal A(\l,0)|=2L_{r-3}$, where $L_k$ denotes the $k^\text{th}$ Lucas number.
\end{corollary}
The proof of Corollary \ref{LucasD} is analogous to that of Corollary \ref{LucasC}, hence we omit it.

\subsection{Nonzero weight spaces }

\section{A $q$-analog}\label{aqanalog}
The $q$-analog of Kostant's partition function is the polynomial valued function, $\wp_q$, defined on $\mathfrak h^*$ by $\wp_q(\xi)=c_0+c_1q+\cdots+c_k q^k$, where $c_j$= number of ways to write $\xi$ as a nonnegative integral sum of exactly $j$ positive roots, for $\xi\in\mathfrak h^*$. The $q$-analog of Kostant's weight multiplicity formula is defined, in \cite{LL}, as:
\begin{center}
$m_q(\lambda,\mu)=\sum\limits_{\sigma\in W}^{}(-1)^{\ell(\sigma)}\wp_q(\sigma(\lambda+\rho)-(\mu+\rho))$.
\end{center}

It is known that the multiplicity of the zero weight in the representation $L(\l)$ is equal to $1$, see \cite{BZ}. In this section, we give a combinatorial proof of this fact, by proving the following.

The $q$-analog of Kostant's partition function is the polynomial valued function, $\wp_q$, defined on $\mathfrak h^*$ by $\wp_q(\xi)=c_0+c_1q+\cdots+c_k q^k$, where $c_j$= number of ways to write $\xi$ as a non-negative integral sum of exactly $j$ positive roots, for $\xi\in\mathfrak h^*$. The $q$-analog of Kostant's weight multiplicity formula is defined, in \cite{LL}, as:
\begin{center}
$m_q(\lambda,\mu)=\sum\limits_{\sigma\in W}^{}(-1)^{\ell(\sigma)}\wp_q(\sigma(\lambda+\rho)-(\mu+\rho))$.
\end{center}

It is known that the multiplicity of the zero weight in the representation $L(\varpi_1)$ is equal to $1$, see \cite{BZ}. In this section, we give a combinatorial proof of this fact, by proving the following.

\begin{theorem}\label{qmultofzeroinso-odd} Let $r\ge 2$ and let $\varpi_1=\sum_{\alpha\in\Delta}^{}\alpha$ be a fundamental weight of $\mathfrak{so}_{2r+1}$. Then $m_q(\varpi_1,0)=q^r$.
\end{theorem}

Observe that the subset of positive roots of $\mathfrak{so}_{2r+1}$ used to write $\sigma(\varpi_1+\rho)-\rho$, for any $\sigma\in\mathcal A(\varpi_1,0)$, is equal to the set of positive roots of $\mathfrak{sl}_{r+1}$. Therefore, the following lemmas and propositions follow from Lemma 3.1 and Proposition 3.2 in \cite{PH}.

\begin{lemma}\label{numberofelements1}
The cardinality of the sets $\{\sigma\in\mathcal A(\varpi_1,0):\ell(\sigma)=k\text{ and $\sigma$ contains no $s_r$ factor}\}$ and $\{\sigma\in\mathcal A(\varpi_1,0):\ell(\sigma)=k\text{ and $\sigma$ contains an $s_r$ factor}\}$ are $\binom{r-1-k}{k}$ and $\binom{r-2-k}{k}$, respectively. Also $max\{\ell(\sigma):\sigma\in\mathcal A(\varpi_1,0)\text{ and $\sigma$ contains no $s_r$ factor}\}=\lfloor \frac{r-1}{2}\rfloor$ and $max\{\ell(\sigma):\sigma\in\mathcal A(\varpi_1,0)\text{ and $\sigma$ contains an $s_r$ factor}\}=\lfloor \frac{r-2}{2}\rfloor$.
\end{lemma}

\begin{proposition}\label{qanalogofso-odd1}
Let $\sigma\in\mathcal A(\varpi_1,0)$. Then \begin{center}$\wp_q (\sigma(\varpi_1+\rho)-\rho)=\begin{cases}q^{1+\ell(\sigma)}(1+q)^{r-1-2\ell(\sigma)}&\text{if $\sigma$ contains no $s_r$ factor}\\
q^{1+\ell(\sigma)}(1+q)^{r-2-2\ell(\sigma)}&\text{if $\sigma$ contains an $s_r$ factor}.\end{cases}$\end{center}
\end{proposition}

Now can now prove the closed formula for the $q$-multiplicity of the zero weight in $L(\varpi_1)$.

\begin{proof}[Proof of Theorem \ref{qmultofzeroinso-odd}]
Observe that

$m_q(\varpi_1,0)=\displaystyle\sum_{\substack{\sigma\in\mathcal A(\varpi_1,0)\\ \text{ with no $s_r$ factor}}}^{}(-1)^{\ell(\sigma)}\wp_q(\sigma(\varpi_1+\rho)-\rho)+
 \displaystyle\sum_{\substack{\sigma\in\mathcal A(\varpi_1,0)\\ \text{ with an $s_r$ factor}}}^{}(-1)^{\ell(\sigma)}\wp_q(\sigma(\varpi_1+\rho)-\rho).$

By Lemma \ref{numberofelements1}, Proposition \ref{qanalogofso-odd1} and Proposition 3.3 in \cite{PH} it follows that
\begin{align*}
\displaystyle\sum_{\substack{\sigma\in\mathcal A(\varpi_1,0)\\ \text{ with no $s_r$ factor}}}^{}(-1)^{\ell(\sigma)}\wp_q(\sigma(\varpi_1+\rho)-\rho)&=\displaystyle\sum_{k=0}^{\lfloor \frac{r-1}{2}\rfloor}(-1)^k \binom{r-1-k}{k} q^{1+k}(1+q)^{r-1-2k}\\
&=\sum_{i=1}^{r}q^i,\text{ and}
\end{align*}
\begin{align*}
\displaystyle\sum_{\substack{\sigma\in\mathcal A(\varpi_1,0)\\ \text{ with an $s_r$ factor}}}^{}(-1)^{\ell(\sigma)}\wp_q(\sigma(\varpi_1+\rho)-\rho)&=\displaystyle\sum_{k=0}^{\lfloor \frac{r-2}{2}\rfloor}(-1)^{1+k} \binom{r-2-k}{k} q^{1+k}(1+q)^{r-2-2k}\\
&=-\sum_{i=1}^{r-1}q^i.
\end{align*}
Therefore, $m_q(\varpi_1,0)=(q+q^2+\cdots+q^{r-1}+q^r)-(q+q^2+\cdots+q^{r-1})=q^r$.
\end{proof}

\begin{corollary}\label{multofzeroweight}Let $r\ge 2$ and let $\varpi_1=\sum_{\alpha\in\Delta}^{}\alpha$ be a fundamental weight of $\mathfrak{so}_{2r+1}$. Then $m(\varpi_1,0)=1$.
\end{corollary}
\begin{proof}Follows directly from Theorem \ref{qmultofzeroinso-odd}, since $m(\varpi_1,0)=m_q(\varpi_1,0)|_{q=1}=1$.
\end{proof}

\section*{Acknowledgements}The authors thank Anthony Simpson for his programming assistance throughout the  of this project.

\end{document}